\documentclass[12pt]{article}

\RequirePackage{kpfonts}
\RequirePackage[sf,bf,small,raggedright]{titlesec}
\usepackage{textcomp}
\usepackage{amssymb}
\usepackage{bbm}
\usepackage{fancyhdr}
\usepackage{amsmath}
\usepackage{amsbsy,amsthm}
\usepackage{amscd}
\usepackage{latexsym}
\usepackage{graphicx}   
\usepackage{pdfsync}
\usepackage{blkarray}
\usepackage{multirow}
\usepackage{hyperref}

\newcommand{\R}{\mathbb{R}}
\newcommand{\Rd}{\R^d}
\newcommand{\inr}[1]{\left\langle #1 \right\rangle}

\newcommand{\E}{\mathbb{E}}

\newcommand{\eps}{\varepsilon}

\newcommand{\diam}{{\rm diam}}

\newtheorem{lemma}{Lemma}
\newtheorem{theorem}{Theorem}

\numberwithin{equation}{section}

\def \endproof
{{\mbox{}\nolinebreak\hfill\rule{2mm}{2mm}\par\medbreak}}
\def\IND{\mathbbm{1}}

\newcommand{\ol}{\overline}
\newcommand{\wt}{\widetilde}
\newcommand{\wh}{\widehat}
\newcommand{\argmin}{\mathop{\mathrm{argmin}}}

\newcommand{\EXP}{\mathbb{E}}
\newcommand{\PROB}{\mathbb{P}}

\newcommand{\Tr}{\mathrm{Tr}}
\newcommand{\lambdamax}{\lambda_{\text{max}}}

\newcommand{\defeq}{\stackrel{\mathrm{def.}}{=}}

\begin{document}

\title{Sub-Gaussian estimators of the mean of a random vector
\thanks{
G\'abor Lugosi was supported by
the Spanish Ministry of Economy and Competitiveness,
Grant MTM2015-67304-P and FEDER, EU. Shahar Mendelson was supported in part by the Israel Science Foundation.
}
}
\author{
G\'abor Lugosi\thanks{Department of Economics and Business, Pompeu
  Fabra University, Barcelona, Spain, gabor.lugosi@upf.edu}
\thanks{ICREA, Pg. Lluís Companys 23, 08010 Barcelona, Spain}
\thanks{Barcelona Graduate School of Economics}
\and
Shahar Mendelson \thanks{Department of Mathematics, Technion, I.I.T, and Mathematical Sciences Institute, The Australian National University, shahar@tx.technion.ac.il}}

\maketitle

\begin{abstract}
We study the problem of estimating the mean of a random vector $X$
given a sample of $N$ independent, identically distributed points.
We introduce a new estimator that achieves a purely sub-Gaussian
performance under the only condition that the second moment of
$X$ exists. The estimator is based on a novel concept of a multivariate median.

\end{abstract}

\section{Introduction}

In this paper we study the problem of estimating the mean
of a random vector $X$ taking values in $\R^d$.
Denoting the mean by $\mu = \EXP X$,
we assume throughout the paper that the
covariance matrix $\Sigma= \EXP (X-\mu) (X-\mu)^T$ exists.
Suppose that $N$ independent, identically distributed samples
$X_1,\ldots,X_N$ drawn from the distribution of $X$ are available, and one
wishes to estimate the mean vector $\mu$.
An estimator is simply a function of the data that we denote by
$\wh{\mu}_N = \wh{\mu}_N(X_1,\ldots,X_N)$.

There are many possible
ways of measuring the quality of an estimator. The classical statistical
literature tended to focus on risk measures such as
the mean squared error $\EXP \|\wh{\mu}_N-\mu\|^2$.
(Here, and in the rest of the paper,
$\|\cdot\|$ denotes the Euclidean norm in $\Rd$, $S^{d-1}=\{v\in \R^d: \|v\|=1\}$ denotes the Euclidean sphere in $\Rd$
and $\inr{\cdot,\cdot}$ is the usual inner product in $\Rd$.) In this case the sample mean  $\ol{\mu}_N=(1/N)\sum_{i=1}^N X_i$
has a mean squared error equal to $\Tr(\Sigma)/N$
(where $\Tr(\Sigma)$ denotes the trace of the covariance matrix)
and, even though
this estimator is not necessarily optimal even for standard normal vectors---by ``Stein's paradox'', see \cite{JaSt61}---, the order of magnitude of the error cannot be improved in general.

The situation is quite different when one is interested in minimizing
the value $r$ that satisfies
\[
\PROB\left\{ \left\|\wh{\mu}_N-\mu\right\| > r\right\} \le \delta
\]
for some given $\delta >0$. While one may always take $r=\sqrt{\Tr(\Sigma)/(N\delta)}$
for the sample mean, much better dependence on $\delta$ may be achieved
if the
distribution is sufficiently light tailed.
For example, if $X$ has a multivariate normal distribution with
mean $\mu$
and
covariance matrix $\Sigma$, then the sample mean $\ol{\mu}_N$
is also multivariate
normal with mean $\mu$ and covariance matrix $(1/N)\Sigma$ and
therefore, for $\delta \in (0,1)$, with probability at least $1-\delta$,
\begin{equation}
\label{eq:subgauss}
   \left\|\ol{\mu}_N-\mu\right\| \le \sqrt{\frac{\Tr(\Sigma)}{N}}
    + \sqrt{\frac{2\lambdamax \log(1/\delta)}{N}}~,
\end{equation}
where $\lambdamax$ denotes the  largest
eigenvalue of $\Sigma$ (see Hanson and Wright \cite{HaWr71}).
Similar bounds may be proven for the performance of the sample mean
if $X$ has a sub-Gaussian distribution in the sense that for all
unit vectors $v\in S^{d-1}$,
$$
\EXP\exp(\lambda \inr{v,X-\EXP X}) \le \exp(c\lambda^2\inr{v,\Sigma v})
$$
for some constant $c$.

However, when the distribution is not necessarily sub-Gaussian and is possibly
heavy-tailed, one cannot expect such a sub-Gaussian
behavior of the sample mean. Thus, when is it not reasonable to assume
a sub-Gaussian distribution and heavy tails may be a concern, the
sample mean is a risky choice. Indeed, alternative estimators have
been constructed to achieve better performance.

The one-dimensional case (i.e., $d=1$) is quite well understood, see
Catoni \cite{Cat10} and
Devroye, Lerasle, Lugosi, and Oliveira \cite{DeLeLuOl16} for recent accounts.
The so-called \emph{median-of-means} estimator is
a simple and powerful univariate estimator with essentially optimal performance.
This estimate was introduced independently
in various papers, see
Nemirovsky and Yudin \cite{NeYu83},
Jerrum, Valiant, and Vazirani \cite{JeVaVa86},
Alon, Matias, and Szegedy \cite{AMS02}.
The median-of-means estimator partitions the data into $k<N$ blocks of
size $m\approx N/k$ each, computes the sample mean within each block,
and outputs their median.
One may easily show (see, e.g., Hsu \cite{HsuBlog}) that, for any
$\delta \in (0,1)$ if $k=\lceil8\log(1/\delta) \rceil$, then
the resulting estimator $\wh{\mu}_N^{(\delta)}$
satisfies
that, with probability at least $1-\delta$,
\begin{equation}
\label{eq:mom}
\left|\wh{\mu}_N^{(\delta)} - \mu\right| \le 8\sigma \sqrt{\frac{\log(2/\delta)}{N}}
\end{equation}
where $\sigma^2$ denotes the variance of $X$. In other words, in the one-dimensional case, the
median-of-means estimator achieves a sub-Gaussian performance under
the only condition that the variance of $X$ exists.

The median-of-means estimator has been extended to the multivariate
case by replacing the median by its natural
multivariate extension, the so-called ``geometric (or spatial) median''
(i.e., the point that minimizes the sum of the Euclidean distances to the
sample means within each block)
see Lerasle and Oliveira \cite{LerasleOliveira_Robust},
Hsu and Sabato \cite{HsSa16}, Minsker \cite{Min15}.
In particular, Minsker proves that for each $\delta\in (0,1)$
this generalization of the median-of-means estimator $\wt{\mu}_N^{(\delta)}$ is such that,
with probability at least $1-\delta$,
\begin{equation}
\label{eq:minsker}
 \left\|\wt{\mu}_N^{(\delta)}-\mu\right\| \le C\sqrt{\frac{\Tr(\Sigma) \log(1/\delta)}{N}}~,
\end{equation}
where $C$ is  a universal constant. This bound holds under the only assumption
that the covariance matrix exists. However, it does not
quite achieve a sub-Gaussian performance bound that resembles
\eqref{eq:subgauss}.

Joly, Lugosi, and Oliveira \cite{JoLuOl17} made an attempt to
construct a mean estimator with a sub-Gaussian behavior for a large
class of distributions.  They prove that there exists a mean estimator
$\wh{\mu}_n^{(\delta)}$ such that, if the distribution satisfies that
for all $v\in S^{d-1}$
\[
   \EXP\left[ \inr{(X-\mu),v}^4\right] \le K (\inr{v,\Sigma v})^2~,
\]
for some constant $K$,
then for all $N\ge CK \log d \left(d + \log (1/\delta)\right)$, with
probability at least $1-\delta$,
\begin{equation}
 \label{eq:JoLuOlbound}
   \left\|\wh{\mu}_N^{(\delta)}-\mu\right\|
\le C  \left(\sqrt{\frac{\Tr(\Sigma)}{N}}
 + \sqrt{\frac{\lambdamax \log(\delta^{-1}\log d)}{N}} \right)~,
\end{equation}
where again $C$ is a universal constant. This bound resembles the
sub-Gaussian inequality \eqref{eq:subgauss}. However, there are
various caveats: the additional fourth-moment assumption, the
requirement that $N=\Omega(d\log d)$, and, to a lesser extent, the
extra $\log\log d$ term in the bound seem sub-optimal.

The main result of this paper is that there exists a mean estimator that
achieves purely sub-Gaussian performance under the minimal condition
that the covariance matrix exists. More precisely, we prove the
existence of a mean estimator $\wh{\mu}_N^{(\delta)}$ such that, for all
distributions with a finite second moment, for all $N$, with probability at
least $1-\delta$,
\[
   \left\|\wh{\mu}_N^{(\delta)}-\mu\right\|
\le C  \left(\sqrt{\frac{\Tr(\Sigma)}{N}}
 + \sqrt{\frac{\lambdamax \log(2/\delta)}{N}} \right)~,
\]
for an explicit universal constant $C$.

The proposed estimator may be interpreted as a multivariate median-of-means
estimate but with a new notion of a multivariate median which may be
interesting in its own right. The construction of the new estimator
is inspired by the  technique of
``median-of-means tournament'', put forward by the authors in \cite{LuMe16}.

In the next section we present the proposed estimator and the
performance bound. In Section \ref{sec:proof} we present the proofs.
We finish the paper by remarks about the computation of the estimator.

\section{The estimator}

Here we introduce the proposed mean estimator.
Recall that we are given an i.i.d.\ sample $X_1,\ldots,X_N$
of random vectors in $\Rd$.
As in the case of the median-of-means
estimator, we start by partitioning the set
$\{1,\dots,N\}$ into $k$ blocks $B_1,\ldots,B_k$, each of
size
$|B_j|\geq m \defeq \lfloor N/k\rfloor$, where $k$ is a parameter
of the estimator whose value depends on the desired confidence level,
as specified below.
In order to simplify the presentation, in the rest of the paper, without loss of generality, we assume that $N$ is divisible by $k$ and therefore $|B_j|=m$
for all $j=1,\ldots,k$.

Define the sample mean within each block by
\[
Z_j=\frac{1}{m}\sum_{i\in B_j}X_i~.
\]
For each $a\in \R^d$, let
\begin{equation}
\label{eq:sa}
    S_a=\left\{ x\in \R^d: \exists J\subset [k]: |J|> k/2 \ \text{such that} \
      \min_{j\in J} \left(\|Z_j-x\| - \|Z_j-a\|\right) >0 \right\}
\end{equation}
and define the mean estimator by
\[
\wh\mu_N \in \argmin_{a\in \Rd} \diam(S_a^c)~.
\]
Thus, $\wh\mu_N$ is chosen to minimize, over all $a\in \R^d$,
the diameter of the complement of set $S_a$ defined as the set of points $x\in \R^d$ for which
$\|Z_j-x\| > \|Z_j-a\|$ for the majority of the blocks, and if there are several minimizers, one may pick any one of them.

Note that the minimum is always achieved. This follows from the fact
that $\diam(S_a^c)$ is a continuous function of $a$
(since, for each $a$, $S_a^c$ is the intersection of a finite union of
closed balls, and the centers and radii of the closed balls are continuous in $a$).

One may interpret $\argmin_{a\in \Rd} \diam(S_a^c)$ as a new multivariate notion
of the median of $Z_1,\ldots,Z_k$. Indeed, when $d=1$, it is a particular
choice of the median and
the proposed estimator coincides
with the median-of-means estimator.

The main result of this paper is the following performance bound:

\begin{theorem} \label{thm:main}
Let $\delta \in (0,1)$ and consider the mean
estimator $\wh{\mu}_N$ with parameter $k= \lceil 360\log(2/\delta)\rceil $.
If $X_1,\ldots,X_N$ are i.i.d.\ random vectors in $\Rd$ with mean
$\mu\in \Rd$ and covariance matrix $\Sigma$, then for all $N$, with probability at least $1-\delta$,
\[
   \left\|\wh{\mu}_N-\mu\right\|
\le  2\max\left(400 \sqrt{\frac{\Tr(\Sigma)}{N}},
  240 \sqrt{\frac{\lambdamax \log(2/\delta)}{N}} \right)~.
\]
\end{theorem}

Thus, the proposed estimator achieves a purely sub-Gaussian performance
under minimal conditions. Just like in the case of the median-of-means
estimator for the univariate case, the estimator depends on the desired
level of confidence $\delta$. As it is shown in \cite{DeLeLuOl16}, such
a dependence cannot be avoided without imposing additional conditions
on the distribution. However, following the route laid down in \cite{DeLeLuOl16},
one may construct sub-Gaussian estimators that work for a wide range of
confidence levels simultaneously under more assumptions on the distribution.
Since this issue is beyond the scope of this paper and will not be
pursued further here.

Just like Minsker's bound \eqref{eq:minsker}---but unlike the bound
\eqref{eq:JoLuOlbound}---, the performance bound of
Theorem \ref{thm:main} is ``infinite-dimensional'' in the sense that
the bound does not depend on the dimension $d$ explicitly. Indeed,
the same estimator may be defined for Hilbert-space valued random vectors
and Theorem \ref{thm:main} remains valid as long as $\Tr(\Sigma)=\EXP\|X-\mu\|^2$
is finite.

Theorem \ref{thm:main} is an outcome of the following observation which is of interest in its own right on the geometry of a typical collection $\{X_1,...,X_N\}$.
\begin{theorem} \label{thm:geometry}
Using the same notation as above and setting
\[
r=\max\left(400 \sqrt{\frac{\Tr(\Sigma)}{N}},
  240 \sqrt{\frac{\lambdamax \log(2/\delta)}{N}} \right)~,
\]
with probability at least $1-\delta$, for any $a \in \Rd$ such that $\|a-\mu\| \geq r$, one has $\|Z_j-a\| > \|Z_j-\mu\|$ for more than $k/2$ indices $j$.
\end{theorem}
Theorem \ref{thm:geometry} implies that for a `typical' collection $X_1,...,X_N$, $\mu$ is closer to a majority of the $Z_j$'s when compared to any $a \in \Rd$ that is sufficiently far from $\mu$. Obviously, for an arbitrary collection $x_1,...,x_N \subset \Rd$ such a point need not exist, and it is rather surprising that for a typical i.i.d.\ configuration, this property is satisfied by $\mu$.

The fact that Theorem \ref{thm:geometry} implies Theorem \ref{thm:main} is straightforward. Indeed, Theorem \ref{thm:geometry} implies that $\diam(S_\mu^c) \leq 2r$ and that if $\|a-\mu\| \geq r$, then $\mu \in S_a^c$. By the definition of $S_a$, one always has $a \in S_a^c$, and thus if  $\|a-\mu\| > 2r$ then $\diam(S_a^c) > 2r$. Therefore, the minimizer $\wh{\mu}$ must satisfy that $\|\wh{\mu}-\mu\| \leq 2r$, as required.

We do not claim that the values of the constants appearing in Theorem \ref{thm:main} are optimal. They were obtained with the goal of making the proof transparent, nothing more, and it is likely that they may be improved by more careful calculations.

The proof of Theorem \ref{thm:geometry} is based on the idea of  ``median-of-means tournaments'' which was introduced
by Lugosi and Mendelson \cite{LuMe16} is the context of regression function
estimation.

\section{Proof}
\label{sec:proof}

The proof of Theorem \ref{thm:geometry} is based on the following idea.
The mean $\mu$ is the minimizer of the function $f(x)= \EXP \|X-\mu\|^2$.
A possible approach is to use the available data to guess, for
any pair $a,b\in\R^d$, whether $f(a)< f(b)$. To this end, we may
set up a ``tournament'' as follows.

Recall that $[N]$ is partitioned into $k$ disjoint blocks $B_1,\ldots,B_k$
of size $m=N/k$. For $a,b \in \R^d$, we say that
$a$ \emph{defeats} $b$ if
\[
\frac{1}{m} \sum_{i \in B_j} \left(\|X_i-b\|^2 - \|X_i-a\|^2\right) > 0
\]
on more than $k/2$  blocks $B_j$. The main technical lemma is the following.

\begin{lemma}
\label{lem:tournament}
Let $\delta\in (0,1)$, $k= \lceil 360\log(2/\delta)\rceil$, and define
\[
r=\max\left(400 \sqrt{\frac{\Tr(\Sigma)}{N}},
  240 \sqrt{\frac{\lambdamax \log(2/\delta)}{N}} \right)~.
 \]
 With probability at least $1-\delta$,
$\mu$ defeats all $b\in \Rd$ such that
$\|b-\mu\| \geq r$.
\end{lemma}

\begin{proof}
Note that
\[
\|X_i-b\|^2 - \|X_i-\mu\|^2 = \|X_i-\mu+\mu-b\|^2 - \|X_i-\mu\|^2 = -2\inr{X_i-\mu,b-\mu}+\|b-\mu\|^2~,
\]
set $\ol{X}=X-\mu$ and put $v=b-\mu$. Thus, for a fixed $b$ that satisfies $\|b-\mu\| \geq r$, $\mu$ defeats $b$ if
\[
-\frac{2}{m}\sum_{i \in B_j} \inr{\ol{X}_i,v}+\|v\|^2>0
\]
on the majority of blocks $B_j$.

Therefore, to prove our claim we need that, with probability at least $1-\delta$, for every $v \in \R^d$ with $\|v\| \geq r$,
\begin{equation} \label{eq:basic}
-\frac{2}{m}\sum_{i \in B_j} \inr{\ol{X}_i,v}+\|v\|^2>0
\end{equation}
for more than $k/2$ blocks $B_j$.
Clearly, it suffices to show that \eqref{eq:basic} holds when $\|v\|=r$.

Consider a fixed $v \in \R^d$ with $\|v\|=r$. By Chebyshev's inequality, with probability at least $9/10$,
\[
\left|\frac{1}{m}\sum_{i \in B_j} \inr{\ol{X}_i,v}\right| \leq \sqrt{10} \sqrt{\frac{\EXP \inr{\ol{X},v}^2}{m}} \leq \sqrt{10} \|v\|\sqrt{\frac{\lambdamax}{m}}~,
\]
where recall that $\lambdamax$ is the largest eigenvalue of the covariance matrix of $X$. Thus, if
\begin{equation} \label{eq:r-cond-1}
r=\|v\| \ge 4\sqrt{10} \sqrt{\frac{\lambdamax}{m}}
\end{equation}
then with probability at least $9/10$,
\begin{equation} \label{eq:basic-1}
-\frac{2}{m}\sum_{i \in B_j} \inr{\ol{X}_i,v} \geq \frac{-r^2}{2}.
\end{equation}
Applying a standard binomial tail estimate, we see that
\eqref{eq:basic-1} holds for a single $v$ with probability at least
$1-\exp(-k/180)$ on at least $8/10$ of the blocks $B_j$.

Now we need to extend the above from a fixed vector $v$ to all
vectors with norm $r$. In order to show that
\eqref{eq:basic-1} holds simultaneously for all $v\in r\cdot S^{d-1}$
on at least $7/10$ of the blocks $B_j$, we first consider a
maximal $\eps$-separated set $V_1 \subset r\cdot S^{d-1}$
with respect to the $L_2(X)$ norm. In other words, $V_1$ is a subset of
$r\cdot S^{d-1}$ of maximal cardinality such that for all $v_1,v_2\in V_1$,
$\|v_1-v_2\|_{L_2(X)}= \inr{v_1-v_2,\Sigma(v_1-v_2)}^{1/2} \ge \eps$. We may estimate this cardinality
by the ``dual Sudakov'' inequality (see \cite{LeTa91} and also \cite{Ver09} for a version with the specified constant),
which implies that the cardinality of $V_1$ is bounded by
\[
\log |V_1| \leq \left(\frac{\E\left[\inr{G,\Sigma G}^{1/2}\right]}{4\eps/r}\right)^2~,
\]
where $G$ is a standard normal vector in $\R^d$.
Notice that for any $a \in \Rd$, $\EXP_X \inr{a,X}^2 = \inr{a,\Sigma a}$, and therefore,
\begin{eqnarray*}
\E\left[\inr{G,\Sigma G}^{1/2}\right]
& = & \E_G \left[\left(\E_X
    \left[\inr{G,\ol{X}}^2\right] \right)^{1/2}\right]
\leq \left(\E_X \E_G \left[\inr{G,\ol{X}}^2 \right] \right)^{1/2} \\
& = & \left(\E\left[ \left\|\ol{X}\right\|^2\right]\right)^{1/2} = \sqrt{\Tr(\Sigma)}~.
\end{eqnarray*}
Hence, by setting
\begin{equation} \label{eq:mesh}
\eps = 5 r \left(\frac{1}{k}\Tr(\Sigma)\right)^{1/2}~,
\end{equation}
we have $|V_1|\le e^{k/360}$ and thus, by the union bound,
with probability at least $1-e^{-k/360}\ge 1-\delta/2$,
\eqref{eq:basic-1} holds for all $v\in V_1$
on at least $8/10$ of the blocks $B_j$.

Next we check that property \eqref{eq:basic} holds simultaneously for all $x$ with
$\|x\|=r$ on at least $7/10$ of the blocks $B_j$.

For every $x \in r \cdot S^{d-1}$, let $v_x$ be the nearest element to
$x$ in $V_1$ with respect to the $L_2(X)$ norm.
It suffices to show that, with probability at least $1-\exp(-k/200)\ge 1-\delta/2$,
\begin{equation} \label{eq:osc}
\sup_{x \in r\cdot S^{d-1}} \frac{1}{k} \sum_{j=1}^k \IND_{\{|m^{-1}\sum_{i \in B_j} \inr{\ol{X}_i,x-v_x}| \geq r^2/4\}} \leq \frac{1}{10}~.
\end{equation}
Indeed, on that event it follows that for every $x \in r\cdot S^{d-1}$, on at least $7/10$ of the coordinate blocks $B_j$, both
$$
-\frac{2}{m} \sum_{i \in B_j} \inr{\ol{X}_i,v_x} \geq \frac{-r^2}{2} \ \ \ {\rm and} \ \ \ 2\left|\frac{1}{m} \sum_{i \in B_j} \inr{\ol{X}_i,x}-\frac{1}{m} \sum_{i \in B_j} \inr{\ol{X}_i,v_x}\right| < \frac{r^2}{2}
$$
hold and hence, on those blocks, $-\frac{2}{m} \sum_{i \in B_j}
\inr{\ol{X}_i,x} +r^2>0$ as required.

It remains to prove \eqref{eq:osc}. Observe that
\[
\frac{1}{k} \sum_{j=1}^k \IND_{\{|m^{-1}\sum_{i \in B_j} \inr{\ol{X}_i,x-v_x}| \geq r^2/4\}} \leq \frac{4}{r^2} \frac{1}{k}\sum_{j=1}^k \left|\frac{1}{m}  \sum_{i \in B_j} \inr{\ol{X}_i,x-v_x} \right|~.
\]
Since $\|x-v_x\|_{L_2(X)} = (\E \inr{X,x-v_x}^2)^{1/2} \leq \eps$ it follows that for every $j$
\[
\E \left|\frac{1}{m}  \sum_{i \in B_j} \inr{\ol{X}_i,x-v_x} \right|
\leq \sqrt{\frac{\EXP\left[\inr{\ol{X},x-v_x}^2\right]}{m}}
\leq \frac{\eps}{\sqrt{m}}~,
\]
and therefore,
\begin{eqnarray*}
\lefteqn{
  \E \sup_{x \in r\cdot S^{d-1}} \frac{1}{k} \sum_{j=1}^k \IND_{\{|m^{-1}\sum_{i \in B_j} \inr{\ol{X}_i,x-v_x}| \geq r^2/4\}} } \\
& \leq &
\frac{4}{r^2} \E \sup_{x \in r\cdot S^{d-1}} \frac{1}{k}\sum_{j=1}^k \left(\left|\frac{1}{m}  \sum_{i \in B_j} \inr{\ol{X}_i,x-v_x} \right| - \E \left|\frac{1}{m}  \sum_{i \in B_j} \inr{\ol{X}_i,x-v_x} \right|\right) +  \frac{4\eps}{r^2\sqrt{m}} \\
& \defeq & (A)+(B)~.
\end{eqnarray*}
To bound $(B)$, note that, by \eqref{eq:mesh},
\[
\frac{4\eps}{r^2\sqrt{m}} = 20 \left(\frac{\Tr(\Sigma)}{N}\right)^{1/2} \cdot \frac{1}{r} \leq \frac{1}{20}
\]
provided that
\[
r \geq 400 \left(\frac{\Tr(\Sigma)}{N}\right)^{1/2}.
\]
Turning to $(A)$, by symmetrization, contraction for Bernoulli processes and de-symmetrization (see, e.g., \cite{LeTa91}), and noting that $\|x-v_x\| \le 2r$, we have
\begin{align*}
(A) & \leq \frac{8}{r^2} \E \sup_{x \in r\cdot S^{d-1}} \left|\frac{1}{N}\sum_{i=1}^N  \inr{\ol{X}_i,x-v_x}\right|  \leq \frac{16}{r} \E \sup_{\{t:\|t\|\le 1\}} \left|\frac{1}{N} \sum_{i=1}^N \inr{\ol{X}_i,t}\right|
\\
& \leq  \frac{16}{r} \cdot \frac{\E\left\|\ol{X}\right\|}{\sqrt{N}}
= \frac{16}{r} \left(\frac{\Tr(\Sigma)}{N} \right)^{1/2} \leq \frac{1}{20}
\end{align*}
provided that $r \geq 320 \left(\frac{\Tr(\Sigma)}{N}\right)^{1/2}.$

Thus, for
\[
Y=  \sup_{x \in r\cdot S^{d-1}} \frac{1}{k} \sum_{j=1}^k \IND_{\{|m^{-1}\sum_{i \in B_j} \inr{\ol{X}_i,x-v_x}| \geq r^2/4\}}~,
\]
we have proved that $\EXP Y \le 1/20$. Finally, in order to prove \eqref{eq:osc}, it suffices to prove that,
$\PROB\{ Y > \EXP Y + 1/20\}\le e^{-k/200}$, which follows from the bounded differences inequality (see, e.g., \cite[Theorem 6.2]{BoLuMa13}).
\end{proof}

\subsection*{Proof of Theorem \ref{thm:geometry}}

Theorem \ref{thm:geometry} is easily derived from Lemma \ref{lem:tournament}.
Fix a block $B_j$, and recall that $Z_j=\frac{1}{m}\sum_{i \in B_j}X_i$.
Let $a,b \in \R^d$. Then
\begin{eqnarray*}
\frac{1}{m}\sum_{i \in B_j} \left(\|X_i-a\|^2- \|X_i -b\|^2 \right)
& = & \frac{1}{m}\sum_{i \in B_j} \left( \|X_i-b-(a-b)\|^2- \|X_i
  -b\|^2 \right)
\\
& = & -\frac{2}{m}\sum_{i \in B_j} \inr{X_i-b,a-b} + \|a-b\|^2 = (*)
\end{eqnarray*}
Observe that $-\frac{2}{m}\sum_{i \in B_j} \inr{X_i-b,a-b} = -2\inr{\frac{1}{m}\sum_{i \in B_j} X_i -b,a-b}=-2\inr{Z_j-b,a-b}$, and thus
\begin{eqnarray*}
(*) &= & -2\inr{Z_j-b,a-b} + \|a-b\|^2 \\
& = & -2\inr{Z_j-b,a-b} + \|a-b\|^2 + \|Z_j-b\|^2 - \|Z_j-b\|^2 \\
& = & \|Z_j-b - (a-b)\|^2 - \|Z_j-b\|^2 =
\|Z_j-a\|^2-\|Z_j-b\|^2~.
\end{eqnarray*}
Therefore, $(*)>0$ (i.e., $b$ defeats $a$ on block $B_j$)
if and only if $\|Z_j-a\| > \|Z_j-b\|$.

Recall that Lemma \ref{lem:tournament} states that, with probability at least $1-\delta$,
if $\|a-\mu\| \geq r$ then on more than $k/2$ blocks $B_j$, $\frac{1}{m}\sum_{i \in B_j} \left(\|X_i-a\|^2- \|X_i -\mu\|^2 \right) >0$, which, by the above argument, is the same as saying that for at least $k/2$ indices $j$,
$\|Z_j-a\| > \|Z_j-\mu\|$.
\endproof

\section{Computational considerations}

The problem of computing various notions of multivariate medians has been thoroughly studied in computational geometry and we refer to Aloupis \cite{Alo06} for a survey on this topic. For example, computing the geometric median---and therefore
the multivariate median-of-means estimator proposed by Hsu and Sabato \cite{HsSa16} and Minsker \cite{Min15}---involves solving a convex optimization problem. Thus, the geometric median may be approximated efficiently, see \cite{CoLeMiPaSi16} for the most recent result and for the rich history of the problem.


In contrast, efficiently computing, or even approximating,
the multivariate median proposed in this paper appears to be a
nontrivial challenge.

A possible approach for computing a mean estimator
that approximates $\wh{\mu}_N$ is based on a variant of
a coordinate descent algorithm that works roughly as follows: starting
with an arbitrary line in $\R^d$, one may discretize, with mesh $O(r)$,
the segment on the line
that supports the convex hull of $Z_1,\ldots,Z_k$. Then one uses pairwise
comparisons of the discretized values, using
the median-of-means
estimate, to find a point that defeats every other
candidate on the line that is at least distance $2r$ apart from it. (With a minor
adjustment of our arguments above one may prove that such a point always
exists.) Then take a line that is orthogonal to the first line and contains
the ``winner'' and repeat the search on that line. Continue for $d$ steps.
One may prove that the point $\wt{\mu}_N$ obtained at the final step is such that,
with probability at least $1-\delta$, $\|\wt{\mu}_N-\mu\|_{\infty}\le Cr$
for a numerical constant $C$. This algorithm runs in time
quadratic in $1/r$ and linear in $d$ but unfortunately it only guarantees
closeness to the true mean in the $\ell_{\infty}$ sense. If one replaces
orthogonal lines by random ones and keeps repeating the procedure,
one eventually achieves the desired guarantee in the Euclidean distance. However,
one needs to consider exponentially many (in $d$) directions to
approach $\mu$ with the desired precision. Note that such algorithms use $r$ as
an input parameter. Naturally, the value of $r$ is not known but the
algorithm is guaranteed to work well as long as the true value of $r$ is
larger that the prior guess.

Another possibility is to start with computing the geometric
median $\wt{\mu}^{(\delta)}$ of the $Z_j$. By \eqref{eq:minsker}, one may now
restrict search to a ball of radius at most $r\sqrt{\log(1/\delta)}$.
By exhaustively searching through this ball (after appropriately discretizing),
one finds an estimate with the desired properties in additional time
of order $\log^d(1/\delta)$. However, this is surely unrealistic
in most interesting cases.

We leave the question of efficiently computing the proposed mean estimate
(or another one with
sub-Gaussian performance guarantees) as an interesting research problem.


\end{document}